\documentclass[journal]{IEEEtran}

\usepackage{amssymb}
\usepackage[ruled,vlined,linesnumbered]{algorithm2e}
\usepackage{amsthm}
\usepackage{amsfonts}
\usepackage{amsmath}
\usepackage{booktabs}
\usepackage{bm}
\usepackage{color}
\usepackage{cite}
\usepackage{comment}
\usepackage{epsfig}
\usepackage{euscript}
\usepackage{fancybox}
\usepackage{graphicx}
\usepackage{hyperref}
\usepackage{lipsum}
\usepackage{pgf}
\usepackage{psfrag}
\usepackage{pifont}
\usepackage{mathrsfs}
\usepackage{multirow}
\usepackage{slashbox}
\usepackage{stfloats}
\usepackage{setspace}
\usepackage{stfloats}
\usepackage{tabularx}
\usepackage{textcomp}

\newcolumntype{C}[1]{>{\centering\arraybackslash}p{#1}}

\newcommand{\qP}{{\bf P}}
\newcommand{\qp}{{\bf p}}
\newcommand{\qF}{{\bf F}}

\newcommand{\qw}{{\bf w}}
\newcommand{\qs}{{\bf s}}
\newcommand{\qz}{{\bf z}}

\newtheorem{Lemma}{Lemma}

\SetKwFunction{Dothat}{Do that}

\begin{document}

\title{Electric Vehicle Charge Scheduling Mechanism to Maximize Cost Efficiency and User Convenience}

\author{\IEEEauthorblockN{Hwei-Ming Chung, \IEEEmembership{Student Member, IEEE}, Wen-Tai Li, \IEEEmembership{Member, IEEE}, Chau Yuen, \IEEEmembership{Senior Member, IEEE}, Chao-Kai Wen, \IEEEmembership{Member, IEEE}, and No{\"e}l Crespi, \IEEEmembership{Senior Member, IEEE}}

\thanks{This research was supported in part by the project under Grant NSFC 61750110529 and in part by SUTD-MIT International Design Centre (IDC; idc.sutd.edu.sg).
Also, the work of C.-K. Wen was supported in part by the Ministry of Science and Technology of Taiwan under Grants MOST 103-2221-E-110-029-MY3.}

\thanks{H.-M. Chung was with National Sun Yat-sen University, Kaohsiung 804, Taiwan.
He is now with Department of Informatics, University of Oslo, Oslo 0373, Norway, (e-mail: hweiminc@ifi.uio.no).}

\thanks{W.-T. Li and C. Yuen (corresponding author) are with Department of Engineering Product Development, Singapore  University  of  Technology  and  Design, Singapore 487372 (e-mail: \{wentai\_li, yuenchau\}@sutd.edu.sg).}

\thanks{C.-K. Wen is with the Institute of Communications Engineering, National Sun Yat-sen University, Kaohsiung 804, Taiwan (e-mail: chaokai.wen@mail.nsysu.edu.tw).}

\thanks{No{\"e}l Crespi is with Institut Mines-Telecom, Telecom Sud-Paris, Paris, France (e-mail: noel.crespi@institut-telecom).}

}

\maketitle

\begin{abstract}
This paper investigates the fee scheduling problem of electric vehicles (EVs) at the micro-grid scale.
This problem contains a set of charging stations controlled by a central aggregator.
One of the main stakeholders is the operator of the charging stations, who is motivated to minimize the cost incurred by the charging stations, while the other major stakeholders are vehicle owners who are mostly interested in user convenience, as they want their EVs to be fully charged as soon as possible.
A bi-objective optimization problem is formulated to jointly optimize two factors that correspond to these stakeholders.
An online centralized scheduling algorithm is proposed and proven to provide a Pareto-optimal solution.
Moreover, a novel low-complexity distributed algorithm is proposed to reduce both the transmission data rate and the computation complexity in the system.
The algorithms are evaluated through simulation, and results reveal that the charging time in the proposed method is $30\%$ less than that of the compared methods proposed in the literature.
The data transmitted by the distributed algorithm is $33.25\%$ lower than that of a centralized one.
While the performance difference between the centralized and distributed algorithms is only $2\%$, the computation time shows a significant reduction.

\end{abstract}

\section{Introduction}
\IEEEPARstart{T}{he} multiple issues related to the greenhouse gas emissions of internal combustion engines in conventional vehicles have had a dramatic effect on the development of electric vehicles (EVs).
Recent technological advances in manufacturing efficient batteries, which improve yearly by $20\%$ in terms of cost and $120 \%$ in terms of capacity while providing increased charging rates \cite{2014-EV-everywhere,2014-grid-of-future,EV-outlook}, continue to enhance the attractiveness of EVs.
Accordingly, the demand for EVs has increased by $80 \%$ since 2011 \cite{EV-outlook}.

Two main stakeholders in the problem of charge scheduling are the operator of charging stations and EV owners.
Therefore, two important criteria in evaluating the efficiency of a charge scheduling algorithm are the total charging cost and the users' convenience level.
Many studies have addressed cost minimization \cite{2012-He-EV-selection,2012-Bessa-EV-selection,2014-tang-EV-mincost,2016-shao-EV-mincost, 2016-John-EV-integer} and user-convenience maximization \cite{2012-wen-EV-selection, 2012-peter-EV-selection, 2017-mal-EV-selection} separately, but only a few studies have regarded both as underlying merit factors \cite{2015-stijn-EVselect,2013-Ngu-EV-joint,2016-xu-jier-charg,2016-SI-jier-charg}.
The authors in \cite{2012-He-EV-selection,2012-Bessa-EV-selection,2014-tang-EV-mincost} attempted to minimize charging costs for parking station owners, whereas the authors in \cite{2016-shao-EV-mincost,2016-John-EV-integer} minimized grid-generation cost.
Despite the good results obtained in \cite{2012-He-EV-selection,2012-Bessa-EV-selection,2014-tang-EV-mincost,2016-shao-EV-mincost, 2016-John-EV-integer}, user convenience was ignored.
By contrast, user convenience was regarded as the main objective function of the scheduling problem to be maximized in \cite{2012-wen-EV-selection,2012-peter-EV-selection,2017-mal-EV-selection}.

In \cite{2012-wen-EV-selection, 2012-peter-EV-selection}, user convenience was defined based on the charging states of EVs, and algorithms were proposed to maximize user convenience.
The same problem was considered in \cite{2017-mal-EV-selection}; a distributed algorithm was proposed to solve the problem efficiently.
Although user-convenience maximization was considered in \cite{2012-wen-EV-selection,2012-peter-EV-selection,2017-mal-EV-selection}, the charging cost was not discussed, and no cost-effective strategy was proposed to optimize the costs for charging station or EV owners.

Several studies examined both cost minimization and user-convenience maximization \cite{2015-stijn-EVselect,2013-Ngu-EV-joint,2016-xu-jier-charg,2016-SI-jier-charg}.
These literatures mainly focused on combining the operation of the power grid and the control of EV charging.
The day-ahead scheduling presented in \cite{2015-stijn-EVselect} minimized the charging cost and determined the charging profile by adopting the reinforcement learning method.
The learning index can then be regarded as one kind of user convenience.
In \cite{2013-Ngu-EV-joint}, a solution was proposed to minimize the total home electricity cost while considering users' comfort levels.
References \cite{2016-xu-jier-charg,2016-SI-jier-charg} also provided a control strategy for the grid operator, where user convenience is provided to select the EVs on the basis of the day-ahead decision.
The present work differs from \cite{2015-stijn-EVselect,2013-Ngu-EV-joint,2016-xu-jier-charg,2016-SI-jier-charg} in that it focuses on the online charge scheduling of EVs in multiple parking stations with charging rate limits and load constraints.
Such a scenario is an issue because of the difficulties in reaching a consensus between the two stakeholders.
Furthermore, the inherent difficulties involved in obtaining the future load make optimal scheduling unattainable.

If the future load in a charging station is unknown, then a forecasting method can be used to improve the scheduling performance by not exceeding the available load at each instant.
For this purpose, load forecasting methods, such as those in \cite{2007-taylor-loadpred-EUdata,2012-taylor-loadpred-exp,2013-taylor-loadpredic}, are required.
Load forecasting was utilized in the current study to enhance the performance and accuracy of the proposed scheduling algorithm.

With the large and growing number of EVs on the road, efficient charging of EV batteries has became an important issue.
Coordinated charging is usually preferred over uncoordinated charging, which is known to adversely affect the power grid by increasing the peak load and total cost, and placing stress on distribution transformers \cite{2013-sha-EVinvest}.
The present study considers a coordinated charging process in which a number of parking stations are under the control of a central aggregator (CA).
A sub-aggregator (SA) is installed in the charging station to communicate with the CA.
The CA is responsible for the charge scheduling of EVs by managing the charging rates and the start and finish times of the charging tasks.
This scheduling is conducted by gathering all the required information, including charging requests, arrival times, and deadlines, transmitted by SAs as inputs of the scheduling algorithm.

According to the previous studies, EV charge scheduling can be formulated as a centralized optimization problem in which the CA handles the decision procedure.
Therefore, in this work, charging cost and user convenience are considered in the objective function of the EV charge scheduling problem for two stakeholders.
More specifically, user convenience is defined based on the charging state of an EV and the total charging time.
Given the inherent trade-off between these two factors, optimizing both objectives simultaneously is impossible. 
Therefore, a solution that is sufficiently good with regard to both objectives is explored with the proposed centralized algorithm.
However, the solution cannot handle situations in which a huge amount of EVs are in the charging station due to the high transmission data rate and computation complexity; hence, it cannot be implemented in practical cases.
Instead of determining all the charging decisions in the CA, the SA in each charging station can help to solve a scheduling problem that will overcome this difficulty \cite{2012-He-EV-selection, 2012-wen-EV-selection, 2017-mal-EV-selection, 2016-shao-EV-mincost}.
Therefore, this work also proposes a low-complexity distributed algorithm with outstanding performance as well as the centralized algorithm.

The main contributions of this study are threefold:
\begin{itemize}

\item A bi-objective optimization problem is formulated to consider both charging cost and user convenience as the main performance factors in the scheduling of EVs.
This problem is considered at the micro-grid scale, which consists of a set of SAs controlled by a CA.
With this formulation, we can attempt to reach a consensus between the two stakeholders.

\item A real-time and centralized scheduling algorithm is proposed to obtain a Pareto-optimal solution for the formulated problem.
A detailed explanation of the Pareto-optimal solution is provided in the Appendix.
With the obtained solution, the average charging time can be reduced without affecting the main objective value.

\item A low-complexity distributed algorithm to reduce the transmission data rate and computation complexity is proposed.
Simulation results indicate that the performance degradation with this algorithm is only $2\%$ compared with the centralized algorithm.
However, while the distributed algorithm can make a decision for numerous EVs under one second, this ability does not apply to the centralized algorithm.
Moreover, the transmission data rate is reduced by $33.25\%$ compared with that of the centralized algorithm in the case study.

\end{itemize}

\section{System Model and Problem Formulation}\label{sec:system_model_problem_formulation}

\subsection{System Model}\label{subsec:system_model}
As illustrated in Fig. \ref{fig:cen_System_model}, we consider a micro-grid (e.g., a university or a town) consisting of $M$ parking stations, each of which has an SA installed, and a CA that controls all the SAs.
A total of $N$ EVs are in the grid.
The charge scheduling problem is studied in a time horizon of $T$ with equal length time slots $t=1, 2,\ldots,T$.
When an EV arrives, its corresponding SA gathers all the required information, which is then sent to CA to provide efficient scheduling for the EVs in all of the parking stations.
For EV $i$, the state of charge (SOC) at time slot $t$ is denoted as ${\tt SOC}_{i, t}$; its range is between $0$ (for the empty battery) and $1$ (for the full battery).
The target SOC, denoted as ${\tt SOC}_{i}^{\rm fin}$, indicates the SOC at the finish time of EV $i$. $a_i$ and $r_{i}$ denote the arrival time and deadline, respectively.
We assume that the actual finish time of an EV can be earlier than its deadline but cannot exceed the user-defined deadline.
We use ${N \times T}$  logical (or $0-1$) matrix $\qF$  to keep track of available EVs in the charging stations for different time slots.
The matrix is defined as follows:
\begin{equation}\label{eq:charging_interval}
 f_{i, t} =
\begin{cases}
1,   &\mbox{EV $i$ is in the station at time $t$}, \\
0,   &\mbox{otherwise}.
\end{cases}
\end{equation}

\begin{figure}
\begin{center}
\resizebox{3in}{!}{%
\includegraphics*{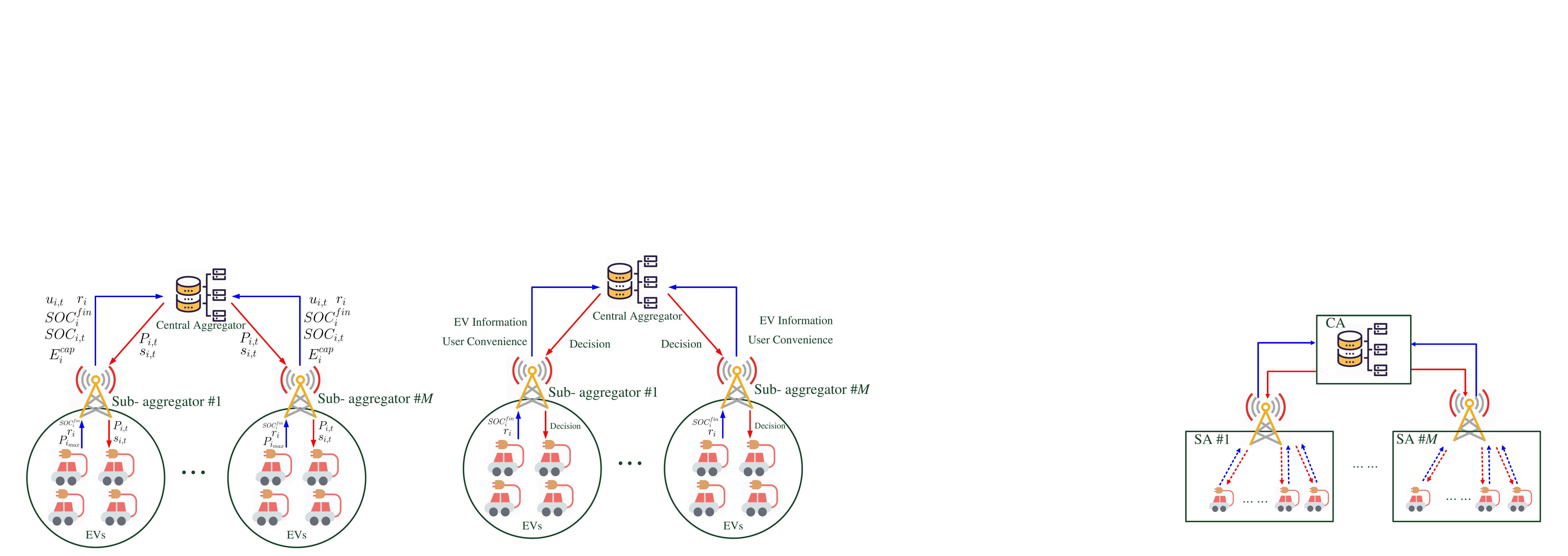} }%
\caption{System model used in this paper. }\label{fig:cen_System_model}
\end{center}
\end{figure}

Given the physical constraints of batteries, the maximum charging rate of EV $i$ is restricted to $P_{i_{\max}}$ in each time slot.
We use $E_{i}^{\rm cap}$, $E_{i}^{\rm ini}$, and $E_{i}^{\rm fin}$ to denote EV $i$'s battery capacity, initial battery energy level at arrival time, and final battery energy level at finish time, respectively.
With these notations, we get $E_{i}^{\rm ini} = {\tt SOC}_{i, a_i}E_{i}^{\rm cap}$.

In time slot $t$, ${\cal H}_{m, t}$ denotes the set of EV indices in the $m$-th charging station.
We use ${\cal H}_{t}= \bigcup^{M}_{m=1} {\cal H}_{m, t}$ to denote the set of EV indices in all charging stations.
Moreover, ${\cal W}_{m, t}$ is the sliding time window of the $m$-th charging station evaluated at time slot $t$; it covers the time slots from $t$ to $t'$, where $t'\geq t$ is the latest deadline of EVs in set ${\cal H}_{m, t}$, i.e., ${\cal W}_{m, t}=\{ t' | t' \geq t ~\&~ t'\leq \max\{r_i | i \in{\cal H}_{m, t} \}\} $.
Similarly, we use ${\cal W}_{t}= \bigcup^{M}_{m=1}  {\cal W}_{m, t} $.

\begin{figure}
\begin{center}
\resizebox{3.2in}{!}{%
\includegraphics*{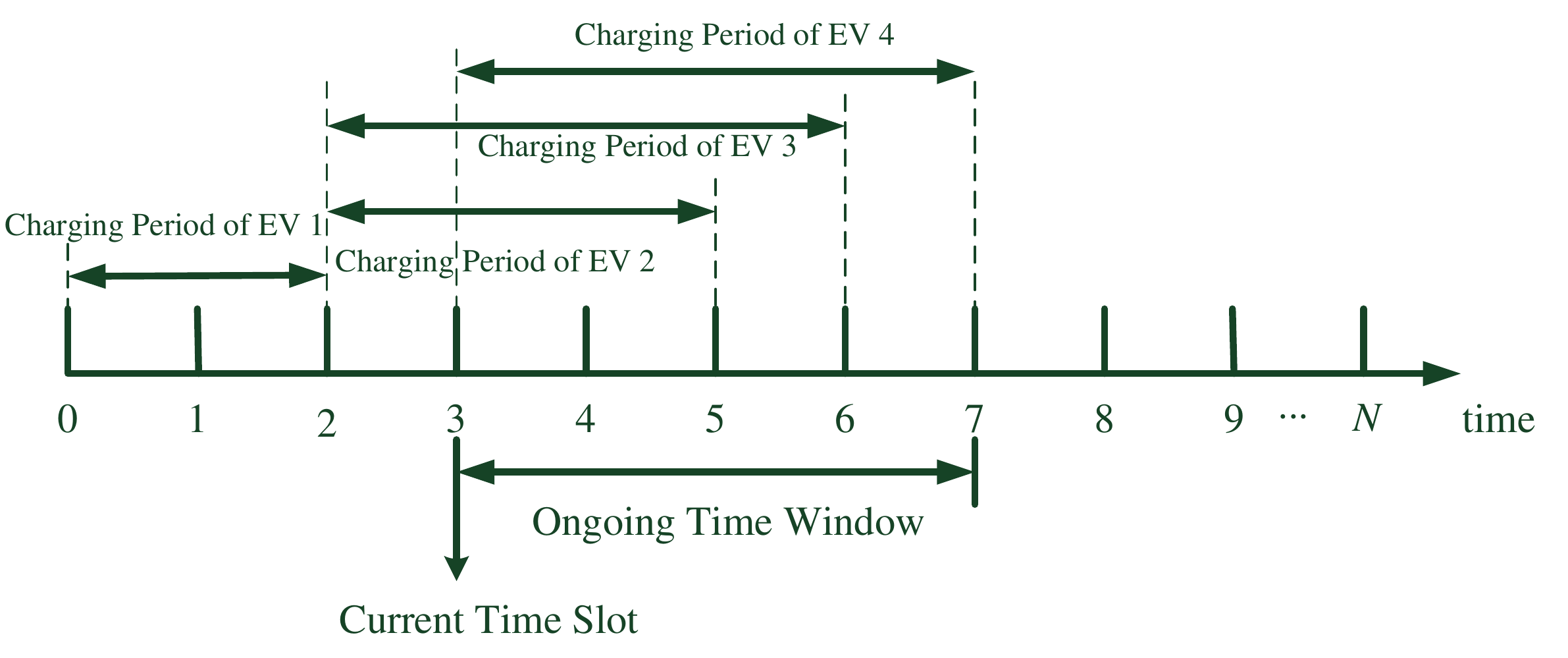} }%
\caption{Illustration of the EV set and the ongoing time window. }\label{fig:EV_time_set}
\end{center}
\end{figure}
Fig.~\ref{fig:EV_time_set} shows an example to provide improved understanding of these notations.
In this example, there are four EVs in the charging stations.
In time slot 3, EV 1 has already left, and EVs 2, 3, and 4 are still in the charging stations.
Therefore, we have ${\cal H}_3 = \{ 2,3,4 \}$. In time slot $t$, the sliding window spans from time slots 3 to 7, i.e., ${\cal W}_3 = \{3, 4, 5, 6, 7 \}$, because
in set ${\cal H}_3$, EV 4 is the last one to leave the station in time slot 7.

\subsection{Charging Cost}
One of the performance factors of interest in this study is the charging cost for the parking station owner.
We adopt the cost function from \cite{2012-He-EV-selection} as follows:
\begin{equation}\label{eq:simple_cost}
	g(z_{t}) = k_{0} + k_{1} z_{t}.
\end{equation}
Function $g(z_{t})$ returns the electricity cost for total load $z_t$ in time slot $t$, where $k_0$ and $k_1$ are constants.
Therefore, the cost for EV charging at time slot $t$ is given by
\begin{equation}\label{eq:int_cost}
	C_{t}\! = \!\int_{L_t^{\rm base}}^{z_{t}}\!\! g(z) \, d z
	       =  k_{0} \Big(z_{t}-L_t^{\rm base}\Big) + \frac{k_{1}}{2} \Big(z_{t}^{2} - {L_t^{\rm base}}^{2}\Big),
\end{equation}
where $L_t^{\rm base}$ is the base load in time slot $t$.
Our first performance metric is as follows:
\begin{equation}\label{eq:obj1}
J_{1}=\sum_{t=1}^{T} C_{t}.
\end{equation}

\subsection{User Convenience}\label{subsec:weight}
To measure the convenience level of users, we define a parameter called user convenience based on a similar definition in \cite{2012-wen-EV-selection}.
The intuition behind user convenience is that EVs need less electricity and that are close to their deadline should have higher charging priority to meet their deadline for their required charges.
Therefore, for EV $i$, we define user convenience as the inverse multiplication of $w^{\star}_{i, t}$ and $w_{i, t}$:
\begin{equation}\label{eq:weight}
u_{i, t} = \frac{1}{ w^{\star}_{i, t} w_{i, t}} .
\end{equation}
Here, the first parameter, $w^{\star}_{i, t}$,  indicates the minimum number of time slots required to finish the charging task of EV $i$ starting from the time slot $t$, which
is defined formally as follows:
\begin{equation} \label{eq:minimum_waiting_time}
w^{\star}_{i, t}   = \frac{ ({\tt SOC}_{i}^{\rm fin} - {\tt SOC}_{i, t} ) E_{i}^{\rm cap}}{P_{i_{\max}}}.
\end{equation}
That is, the charging task of EV $i$ cannot be finished earlier than time slot $t+w^{\star}_{i, t}$ under any feasible scheduling.
The second parameter, $w_{i, t}$, determines the number of remaining time slots to finish the charging task of EV $i$ before its deadline $r_i$:
\begin{equation} \label{eq:parking_time}
w_{i, t} = r_{i} - t .
\end{equation}

A good charging strategy should maximize the sum of all of the users' conveniences.
Therefore, our second performance metric is given by
\begin{equation}\label{eq:obj2}
J_{2}= \sum_{t =1}^{T} \sum_{i = 1}^{N} u_{i, t} .
\end{equation}

\subsection{Problem Formulation}\label{subsec:problem}
We then formulate a bi-objective optimization problem by considering charging cost $J_1$ and user convenience level $J_2$ as follows:
\begin{subequations} \label{eq:total_for}
\renewcommand\minalignsep{0em}
\begin{align}
\mathcal{P} : \quad & \min_{\qz, \qP}\quad \left\{ J_{1}, -J_{2}  \right\} &&    \label{eq:total_obj}\\
\mbox{s.t.} \quad &   L_t^{\rm base} + \sum_{i} P_{i, t} f_{i, t} \leq z_{t} , &&\forall t,  ~\label{eq:total_peak}\\
           &  P_{i_{\min}} \leq P_{i, t}  \leq P_{i_{\max}} ,              &&\forall i, t, \label{eq:total_power_cons}\\
		   &     E_{i}^{\rm ini}+\sum_{t} P_{i, t} f_{i, t} \leq E_{i}^{\rm cap},      &&\forall i, \label{eq:total_no_over}\\
		   &      E_{i}^{\rm fin} \leq E_{i}^{\rm ini} + \sum_{t} P_{i, t} f_{i, t}   ,   &&\forall i. \label{eq:total_meet_target}
\end{align}
\end{subequations}
In constraint (\ref{eq:total_peak}), the summation of the base load and total electricity flowing to the EVs, $L_t^{\rm base} + \sum_{i} P_{i, t} f_{i, t}$, should not be greater than the given total load, $z_{t}$, in time slot $t$.
Constraint (\ref{eq:total_power_cons}) sets the limitation of the charging rate.
Constraints (\ref{eq:total_no_over}) and (\ref{eq:total_meet_target}) ensure that the total electricity in each EV's battery is not greater than that battery's capacity and that it can meet the final charging requirement, respectively.

\section{Online Centralized Solution}\label{sec:algorithm}
Joint minimization of the bi-objective functions (\ref{eq:total_obj}) is impossible because of the inherent trade off between charging cost and user convenience.
In addition, the globally optimal scheduling scheme in (\ref{eq:total_for}) is impractical because the information about EVs that will arrive the charging station in the future is unavailable.
Therefore, we develop an online algorithm to obtain a Pareto-optimal solution for the problem.

\subsection{Problem Transformation}\label{sec:approach}
We solve the problem (\ref{eq:total_for}) by transforming the original formulation into two phases via the lexicographic ordering method.
In the first phase, we ignore user convenience ($J_2$) and identify an optimal solution with respect to electricity cost ($J_1$), by solving the following problem
\begin{subequations}\label{eq:min_cost}
\renewcommand\minalignsep{0em}
\begin{align}
\mathcal{P}_1: \quad &\min_{\qz, \qP}  \quad  \sum_{t \in {\cal W}_{t}} C_{t} && \label{P1} \\
\mbox{s.t.} \quad
&               L_{t}^{\rm base} + \sum_{i} P_{i, t} f_{i, t} \leq z_{t} , && t\in {\cal W}_{t}, i\in {\cal H}_{t},  \label{eq:min_cost_1} \\
&               P_{i_{\min}} \leq P_{i, t}  \leq P_{i_{\max}} ,          && t\in {\cal W}_{t}, i\in {\cal H}_{t}, \label{eq:min_cost_2} \\
&		        E_{i}^{\rm ini}+\sum_{t} P_{i, t} f_{i, t}\leq E_{i}^{\rm cap},   && t\in {\cal W}_{t}, i\in {\cal H}_{t}, \label{eq:min_cost_3}\\
&		        E_{i}^{\rm fin} \leq E_{i}^{\rm ini} + \sum_{t} P_{i, t}f_{i, t}    , && t\in {\cal W}_{t}, i\in {\cal H}_{t}. \label{eq:min_cost_4}
\end{align}
\end{subequations}
In (\ref{eq:min_cost}), the charging rate and load value should be determined to minimize the charging cost.
By solving $\mathcal{P}_1$, we obtain information on optimal load values $z^\star_t$ that can minimize the charging cost.
Given the optimal load values $z^\star_t$, we may have multiple solutions for $P_{i, t}$.
In other words, as long as we fix $z^\star_t$, we can re-schedule the EVs' charging time without changing the charging cost.

In the second phase, we take $z^\star_t$ from Problem $\mathcal{P}_1$ and solve the following  problem.
\begin{subequations}\label{eq:max_us}
\renewcommand\minalignsep{0em}
\begin{align}
\mathcal{P}_2: \quad & \max_{\qP}  \quad \sum_{ t \in {\cal W}_{t}} \sum_{i \in {\cal H}_{t}} u_{i, t}  &&  \label{P2} \\
\mbox{s.t.} \quad
&               L_{t}^{\rm base}  + \sum_{i} P_{i, t} f_{i, t} \leq  z^\star_t ,  &&   t\in {\cal W}_{t}, i\in {\cal H}_{t},  \\
&               (\ref{eq:min_cost_2}), (\ref{eq:min_cost_3}), (\ref{eq:min_cost_4}).
\end{align}
\end{subequations}
Given that $z^\star_t$'s are fixed, the feasible solutions of $\mathcal{P}_2$ are optimal solutions of $\mathcal{P}_1$.
Therefore, among the feasible solutions of Problem $\mathcal{P}_1$, we maximize user convenience.
An optimal solution for $\mathcal{P}_2$ is also an optimal solution for $\mathcal{P}_1$, so this solution is a Pareto-optimal solution for problem (\ref{eq:total_for}).

\subsection{Scheduling Algorithm Description}

Subsequently, we develop an algorithm based on the concept proposed in Section \ref{sec:approach}.
To solve the problem $\mathcal{P}_1$, the algorithm still requires future base load information.
There are many forecasting techniques where can be applied; however, several of them require a large amount of data to process.
In this study, we use the seasonal ARIMA because it is a well-developed and stable algorithm.
If correlation exists between data, this approach demonstrates good performance based on the forecasting result.
Moreover, seasonal ARIMA applies different prediction models and can be found in References \cite{2007-taylor-loadpred-EUdata} and \cite{2013-taylor-loadpredic} in detail.

The core algorithm for solving Problem $\mathcal{P}$ is listed in Algorithm \ref{ago:EV_sel}.
In lines 3--6 of Algorithm \ref{ago:EV_sel}, the required parameters are collected and prepared.
In line 7, CA solves Problem $\mathcal{P}_1$ to obtain optimal charging load value, $z_t^{\star}$, by using the interior-point method\cite{boyd-cvx-book}.
The available power for charging in time slot $t$ is then obtained as follows:
\begin{equation}\label{eq:avai}
L_t^{\rm char}=z_t^{\star} - L_t^{\rm base}.
\end{equation}

By determining the value of $L_t^{\rm char}$ in each time slot, Algorithm \ref{ago:EV_sel} employs a sub-procedure for user convenience maximization (UCM) described in Algorithm  \ref{ago:UCM} to schedule the EVs' charging rate.
To maximize user convenience, all the available power will be used to charge EVs, and EVs with high user convenience values are charged with maximum charging rate.
Specifically, we sort the EVs in ${\cal H}_t$ based on their user convenience, $u_{i, t}$, in a decreasing order.
Next, we examine the sorted list and select as many as possible EVs  to charge in time slot $t$ with their maximum charging rate (except probably the last selected EV), such that the total charging of the selected EVs is equal to the available load $L_t^{\rm char}$.

\begin{algorithm}
\caption{EV Charge Scheduling Algorithm (CSA)}
\label{ago:EV_sel}
 \DontPrintSemicolon
\KwIn{$N$ EVs with their information}
\KwOut{A feasible scheduling for Problem $\mathcal{P}$}
Forecast the base load by the seasonal ARIMA;

\For{t  = 1 \KwTo T }{
   Determine ${\cal H}_t$ and ${\cal W}_t$ of the current time slot $t$\;
   Construct $\qF$ based on (\ref{eq:charging_interval})\;
   All SAs use (\ref{eq:weight}) to determine the user convenience \;
   $E_{i}^{\rm ini} \leftarrow {\tt SOC}_{i, t} E_{i}^{\rm cap}$ \;
   CA solves Problem $\mathcal{P}_1$ using interior-point methods in \cite{boyd-cvx-book} to obtain $z_{t}^\star$\;
   $L_t^{\rm char} \leftarrow z_{t}^\star - L_t^{\rm base}$\;
   Use ${\rm UCM}$ to solve $\mathcal{P}_2$ \;
    Use (\ref{eq:SOC_update}) to update SOC information \;
   }
\end{algorithm}

\begin{algorithm}
\caption{Algorithm UCM$(L_t^{\rm charg})$}
\label{ago:UCM}
 \DontPrintSemicolon
\KwIn{$L_t^{\rm charg}$, $t$ }
\KwOut{Charging decisions for time slot $t$}

Sort EVs in ${\cal H}_t$ based on their user convenience with decreasing order as $e_1,e_2,\dots ,e_{|{\cal H}_t|}$\;
$L_{\textrm{remain}}\leftarrow L_t^{\rm charg}$ \;

$i=1$\;

\While{$L_{\textrm{remain}}>0\ \wedge\ i\leq |{\cal H}_t|$}{
$\Delta=\min \{L_{\textrm{remain}},P_{i_{\max}}\}$

Charge $e_i$ with the rate of $\Delta$

$L_{\textrm{remain}}\leftarrow L_{\textrm{remain}}-\Delta$

$i\leftarrow i+1$}
\end{algorithm}

When the charging decisions are made in time slot $t$, the SOC information needs to be updated as follows:
\begin{equation}\label{eq:SOC_update}
{\tt SOC}_{i, t+1} = {\tt SOC}_{i, t} + \frac{P_{i, t}}{E_{i}^{\rm cap}}.
\end{equation}
If ${\tt SOC}_{i, t+1}$ reaches the target SOC, then the charging task of EV $i$ is done, and time slot $t+1$ is considered  the finish time of the EV.
The proposed algorithm thus determines, the Pareto-optimal solution; the proof is provided in Appendix \ref{subsec:proof_optimal}.

\section{Low-Complexity Distributed Algorithm}\label{sec:Low_com_dis_alg}
In Algorithm \ref{ago:EV_sel}, the CA handles all decision procedures.
However, the computation complexity of implementing the CSA will be exponential growing with the increase of the amount of EVs.
In addition, a CSA typically requires SAs to send all of the information to the CA leading to the heavy transmission data rate.
We therefore develop a low-complexity distributed algorithm in which the transmission data rate can also be reduced.

\subsection{Low-Complexity Charging-Cost Minimization Algorithm}

In Algorithm \ref{ago:EV_sel}, $\mathcal{P}_1$ is solved by using the interior-point method \cite{boyd-cvx-book}.
The computational complexity of this method increases dramatically as the number of EVs in the charging station increases.
A low-complexity algorithm can be derived by the following lemma and with the usage of SAs.

\begin{Lemma}
If all constraints  are removed from (\ref{eq:min_cost}), then an optimal solution for $\mathcal{P}_1$ should satisfy
\begin{equation}
 \sum_{i\in{\cal H}_{t}}  P_{i, t} = \max\big\{c - L_t^{\rm base}, 0 \big\}, ~\forall t,
\end{equation}
where $c= \sum_{t\in{\cal W}_{t}}{\left(\sum_{i\in{\cal H}_{t}}  P_{i, t} + L_t^{\rm base}\right)}/ |{\cal W}_{t}|$ is the average total load, and $|{\cal W}_{t}|$ represents the cardinality of ${\cal W}_{t}$.
\end{Lemma}
\begin{proof}
The proof is in Appendix \ref{subsec:proof_optimal_cost}.
\end{proof}

In the case of Lemma 1, the optimal solution is a flat profile, i.e., $\sum_{i\in{\cal H}_{t}}  P_{i, t} + L_t^{\rm base}$ is a constant for all $t$.
Then, considering the charging requirement constraints, (\ref{eq:min_cost_3}) and (\ref{eq:min_cost_4}), the total power used for charging should meet the demand shown as follows
\begin{equation}\label{eq:total_demand}
\sum_{t\in{\cal W}_{t}} \sum_{i\in{\cal H}_{t}} P_{i, t}  = \sum_{i\in{\cal H}_{t}}  ({\tt SOC}_{i}^{\rm fin} - {\tt SOC}_{i, t})E_{i}^{\rm cap} .
\end{equation}
However, if power constraint (\ref{eq:min_cost_2}) is active, then the flat profile solution does not hold because $c$ cannot exceed the summation of the maximum charging rate of each EV in each time slot.
The number of EVs will change in some time slots so the summation of the maximum charging rate shall vary.
The portion exceeds the summation of charging limitation must be redistributed to other time slots having EVs whose deadlines are earlier than the current time window.
Moreover, the portion doesn't exceed the limitation should remain flat after redistribution in order to reach the minimum cost based on Lemma 1.

Based on the previous discussion, we design an algorithm called low-complexity charging cost minimization (LCCM) algorithm to obtain the solution of $\mathcal{P}_1$.
First, the $m$-th SA calculates time window ${\cal W}_{m, t}$ and total demand
\begin{equation}\label{eq:demand}
d_{m} = \sum_{i \in {\cal H}_{m, t}} ({\tt SOC}_{i}^{\rm fin} - {\tt SOC}_{i, t})E_{i}^{\rm cap},
\end{equation}
then all SAs send ${\cal W}_{m, t}$'s and $d_{m}$'s to CA.
However, CA does not know the exact charging limitation of each EV.
Therefore, CA constructs the time window denoted as $\overline{\cal W}_{t}$ with the intersection of ${\cal W}_{m, t}$
\begin{equation}
\overline{\cal W}_{t} = \bigcap_{m=1}^{M} {\cal W}_{m, t}
\end{equation}
to ensure that these periods have the most number of EVs.
Next, we distribute $d_{m}$ over $\overline{\cal W}_{t}$ based on Lemma 1 as
\begin{equation}\label{eq:lemma1_solution}
c = \frac{ \sum_{m=1}^{M}d_{m} + \sum_{t\in \overline{\cal W}_{t}} L_t^{\rm base}  }{|\overline{\cal W}_{t}|},
\end{equation}
where $|\overline{\cal W}_{t}|$ is the cardinality of $\overline{\cal W}_{t}$.
Finally, $c$ is assigned to $z_{t}^\star$ as the solution of the load value with minimum charging cost.
Because we can ensure there are more EVs staying in the stations at the end of $\overline{\cal W}_{t}$ than the end of ${\cal W}_{t} $, and we distribute the total demand over $\overline{\cal W}_{t}$.
Therefore, the problem of exceeding the summation of the maximum charging rate can be avoided.

\subsection{Low-Complexity Distributed Algorithm Description}

The distributed charge scheduling algorithm (DCSA) listed in Algorithm \ref{ago:dis_EV_sel} also works in two phases.
The first phase still deals with the problem $\mathcal{P}_1$.
In the beginning, the SAs prepare the required parameters and send them to the CA.
Next, the CA applies LCCM algorithm to determine $z_{t}^\star$.
With $z_{t}^\star$ known, the available power can be obtained by (\ref{eq:avai}) and then broadcasted to all SAs.

\begin{algorithm}
\caption{Distributed Charge Scheduling Algorithm (DCSA)}
\label{ago:dis_EV_sel}
 \DontPrintSemicolon
\KwIn{$N$ EVs with their charging information}
\KwOut{A feasible scheduling for Problem $\mathcal{P}$}
Forecast the base load and broadcast to all SAs\;

\For{t  = 1 \KwTo $T$ }{
   SAs determine ${\cal H}_{m, t}$, ${\cal W}_{m, t}$ and $d_{m}$ based on time $t$\;
   SAs send $d_{m} $ and ${\cal W}_{m, t}$ to CA \;
   CA solves Problem $\mathcal{P}_1$ by applying Algorithm LCCM \;
   Broadcast $L_t^{\rm char}$ to all SAs \;
   All SAs use (\ref{eq:weight}) to determine the user convenience and sort them with decreasing order\;
   Use ${\rm DUCM}$ to solve $\mathcal{P}_2$ \;
   Use (\ref{eq:SOC_update}) to update SOC information \;
   }
\end{algorithm}

The algorithm then enters to the second phase when SAs receive the available power from CA.
In the second phase, the SAs have to determine which EVs should be charged through coordination among themselves without any communication from CA.
Motivated by UCM, the smallest value of user convenience should be identified so that EVs with user convenience larger than the smallest value can be charged with the maximum charging rate.
To this end, the algorithm called distributed user convenience maximization (DUCM) provided in Algorithm \ref{ago:dis_EV_sel} and Fig.~\ref{fig:DUCM_model} is designed using bisection method.

\begin{figure}
\begin{center}
\resizebox{3.5in}{!}{%
\includegraphics*{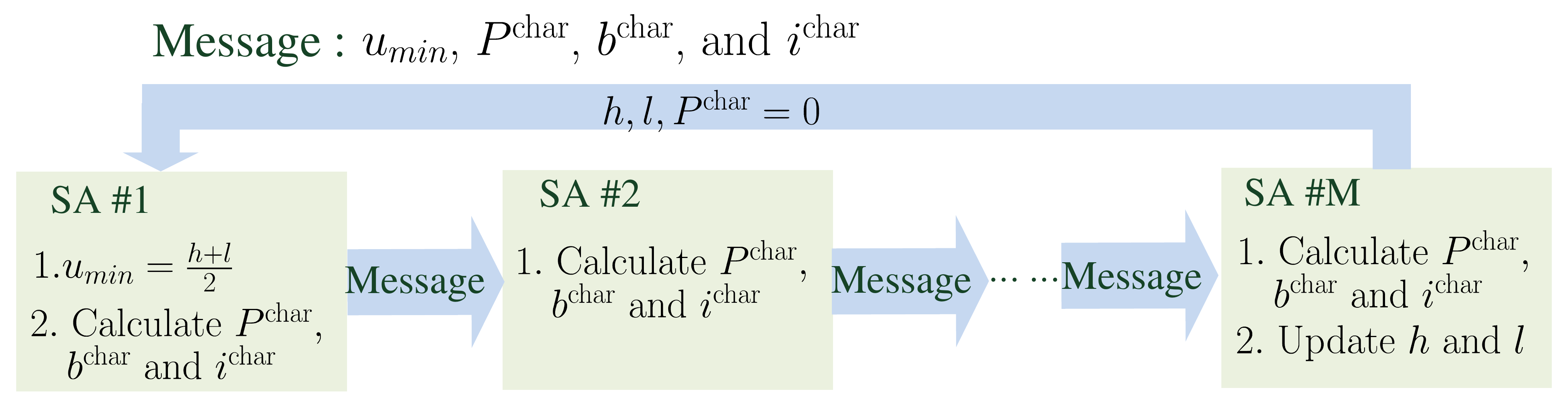} }%
\caption{Iteration process of DUCM}\label{fig:DUCM_model}
\end{center}
\end{figure}

At the beginning, the first SA has to decide the smallest value of the user convenience, $u_{\min}$, by averaging the given upper bound, $h$,  (set to $1$) and the given lower bound, $l$, (set to $0$) of user convenience level.
Every SA then passes the message including $u_{\min}$, $P^{\rm char}$, $b^{\rm char}$, and $i^{\rm char}$ to the next SA.
Here, $P^{\rm char}$ is the summation of the charging power for the EV with user convenience greater than $u_{\min}$,
$b^{\rm char}$ is the largest user convenience under $u_{\min}$, and $i^{\rm char}$ is the corresponding EV number.
Specifically, the three variables are calculated by
\begin{subequations}\label{eq:DUCM_parameter}
\begin{align} 
P^{\rm char} &:= \Big\{\!  P^{\rm char} + \sum_{i} P_{i_{\max}} \Big| i \in {\cal H}_{m, t}, \,u_{i, t} >u_{\min} \!\Big\}, \\
b^{\rm char} &:= \left\{  \max_{i} (u_{i, t}, b^{\rm char}) \Big| i \in {\cal H}_{m, t}, \, u_{i, t} < u_{\min} \right\}, \\
i^{\rm char} &:= \left\{  i | i \in {\cal H}_{m, t}, \, u_{i, t} = b^{\rm char} \right\}.
\end{align}
\end{subequations}
The last SA has to decide new upper or lower bounds based on
\begin{equation} \label{eq:DUCM_new_bound}
\left\{
\begin{array}{ll}
                h \leftarrow u_{\min},      & \mbox{if}~~ L_t^{\rm char} - P^{\rm char}  > P_{i_{\max}},  \\
                l \leftarrow u_{\min},      & \mbox{if}~~ L_t^{\rm char} - P^{\rm char}  < 0,
\end{array}
        \right.
\end{equation}
and send them to the first SA.
Following the iteration process, the smallest value of user convenience level is obtained.
EVs with user convenience larger than $u_{\min}$ are then charged at the maximum charging rate, and the residual power is assigned to EV $ i^{\rm char}$.

The SOC information can be updated with the charging decision using (\ref{eq:SOC_update}).
If ${\tt SOC}_{i, t+1}$ reaches the target SOC, then $t+1$ is considered the finish time of EV $i$.

\begin{algorithm} 
\caption{Algorithm DUCM $(L_t^{\rm char})$}
\label{ago:DUCM}
 \DontPrintSemicolon
\KwIn{$t$, $L_t^{\rm char}$, ${\cal H}_{m, t}$ }
\KwOut{Charging decisions for time slot $t$}
Set upper bound $h = 1$ and lower bound $l=0$ \;

\While {$|h-l| \geq \epsilon$ }{

 first SA: let $ u_{\min} \leftarrow  \frac{1}{2} (h+l)$. Find $ P^{\rm char}$, $b^{\rm char}$, and $i^{\rm char}$ based on (\ref{eq:DUCM_parameter}), and then broadcast to $2^{nd}$ SA \;

\For{k  = 2 \KwTo $M$ }{

	\If{ $k=M$ }{
	Use (\ref{eq:DUCM_new_bound}) to decide the new $h$ and $l$ \;
	Broadcasts $h$, $l$, $P^{\rm char}=0$ to first SA
	}	
	\Else{
	 kth SA finds $u_{\min}$, $P^{\rm char}$, $b^{\rm char}$, and $i^{\rm char}$ based on (\ref{eq:DUCM_parameter}) and broadcasts to (k+1)th SA \;
	 }
}

}

EV with $u_{i, t}$ bigger than $u_{\min}$ is charged with $P_{i_{\max}}$ and EV $i^{\rm char}$ is charged with $(L_t^{\rm char} - P^{\rm char})$
\end{algorithm}

\section{Simulation Results}\label{sec:simulation}
We conduct experiments to evaluate the performance of the proposed algorithms.
Unless otherwise specified, the simulation settings are as follows. We consider a CA that controls six SAs with a size of $ [ N_{1} \, N_{2} \, N_{3} \, N_{4} \, N_{5} \,N_{6}]/N = [5\% ~ 10\% ~  15\% ~ 15\% ~ 20\% ~ 35\% ] $.
Each SA is installed in a parking station.
The total number of EVs is $N$ in the grid, and they are randomly assigned to one of the six parking stations.
The EVs used in the simulation are all Nissan Leaf 2016, each with battery capacity of $30$ kWh and maximum charging rate of $6.6$ kW.
For the pricing model in (\ref{eq:simple_cost}), the constants $k_{0}$ and $k_{1}$ are respectively set to  $10^{-4}$ C\$ and $1.2\times 10^{-4}$ C\$/kWh, as similar setting in \cite{2012-He-EV-selection}.
The time horizon is divided into 96 time slots with a length of 15 minutes to represent a 24-hour period.
The EV arrival times and deadlines are generated randomly between ${18:00}$ and ${07:00}$.
The initial SOC values are randomly and uniformly generated from the interval $[0,1]$, and the target SOC is set to $1$.
The base load information is based on the result of the load forecasting technique. 
We change the unit of the base load to 
\begin{equation}
L_t^{\rm base} = \frac{ L_{t}^{\rm fore} \times L^{\rm peak} } {  \max_{t}(L_{t}^{\rm fore}) }  , 
\end{equation}
where $L^{\rm peak}$ and $L_{t}^{\rm fore} $ are the peak load of a day and the forecasted load in time slot $t$, respectively.
The peak load settings under different numbers of EVs are listed in Table \ref{tb:charging_cost_setting}.
The peak constraint denotes the limitation of $z_{t}$, which is employed for comparison with that in \cite{2012-wen-EV-selection}.
The forecasting error is based on the mean absolute percentage error (MAPE) as follows:
\begin{equation} \label{eq:MAPE}
 {\rm MAPE} = \frac{1}{T} \sum_{t=1}^{T} \left| \frac{L_{t}^{\rm fore} - L_{t}^{\rm act} } { L_{t}^{\rm act} } \right|,
\end{equation}
where $L_{t}^{\rm act}$ is the actual load in time slot $t$.
The simulations for computation speed were conducted with an Intel i7-3770 computer with 3.4\,GHz CPU and 8\,GB RAM.
The error tolerance level for DUCM is set to $10^{-4}$.
All of the simulation results were collected from $100$ Monte Carlo simulations.

\begin{table}\footnotesize
\renewcommand{\arraystretch}{1.1}
\begin{center}
\caption{The parameter setting for different number of EVs}\label{tb:charging_cost_setting}
\begin{tabular}{|C{3.15cm}|C{0.5cm}|C{0.6cm}|C{0.6cm}|C{0.6cm}|C{0.75cm}|c|c|c|}
\hline
EV Number        & 100 & 200 & 300 & 400 & 2000 \\
\cline{1-6}
$L^{\rm peak} $  (Kw)       & 400 & 800 &  1200     &   1600   &  8000 \\
\cline{1-6}
Peak Constraint (Kw) & 800 & 1200 & 1600      &  2000    &  11000 \\
\hline
\end{tabular}
\end{center}
\end{table}

In the simulations, we compare four algorithms, namely CSA (the proposed centralized method), DCSA  (the proposed low-complexity distributed method), the cost minimization algorithm in \cite{2012-He-EV-selection}, and the user convenience maximization algorithm in \cite{2012-wen-EV-selection}.
To solve the optimization problem in \cite{2012-He-EV-selection} and $\mathcal{P}_1$ in CSA, we use CVX \cite{CVX}, a package intended to solve convex programs.

\subsection{Forecasting Accuracy}

We compare the forecasting method (seasonal ARIMA method in \cite{2007-taylor-loadpred-EUdata}) that we used in our scheduling algorithm with the previous day average method used in \cite{2012-He-EV-selection}.
We use the actual load information in France between 11/30/2015 and 12/16/2015 \cite{Load_data} to forecast the load from 12/17 to 12/18 noon.
As shown in Fig.~\ref{fig:load_1217_1218}, the ARIMA method forecasts the future load with high accuracy.
The forecasting errors for seasonal ARIMA and previous days' average method used in \cite{2012-He-EV-selection} are $1.79 \%$ and $4.52\%$, respectively. The result indicates a $2.73 \%$  improvement in the ARIMA method.

\begin{figure}
\begin{center}
\resizebox{2.8in}{!}{%
\includegraphics*{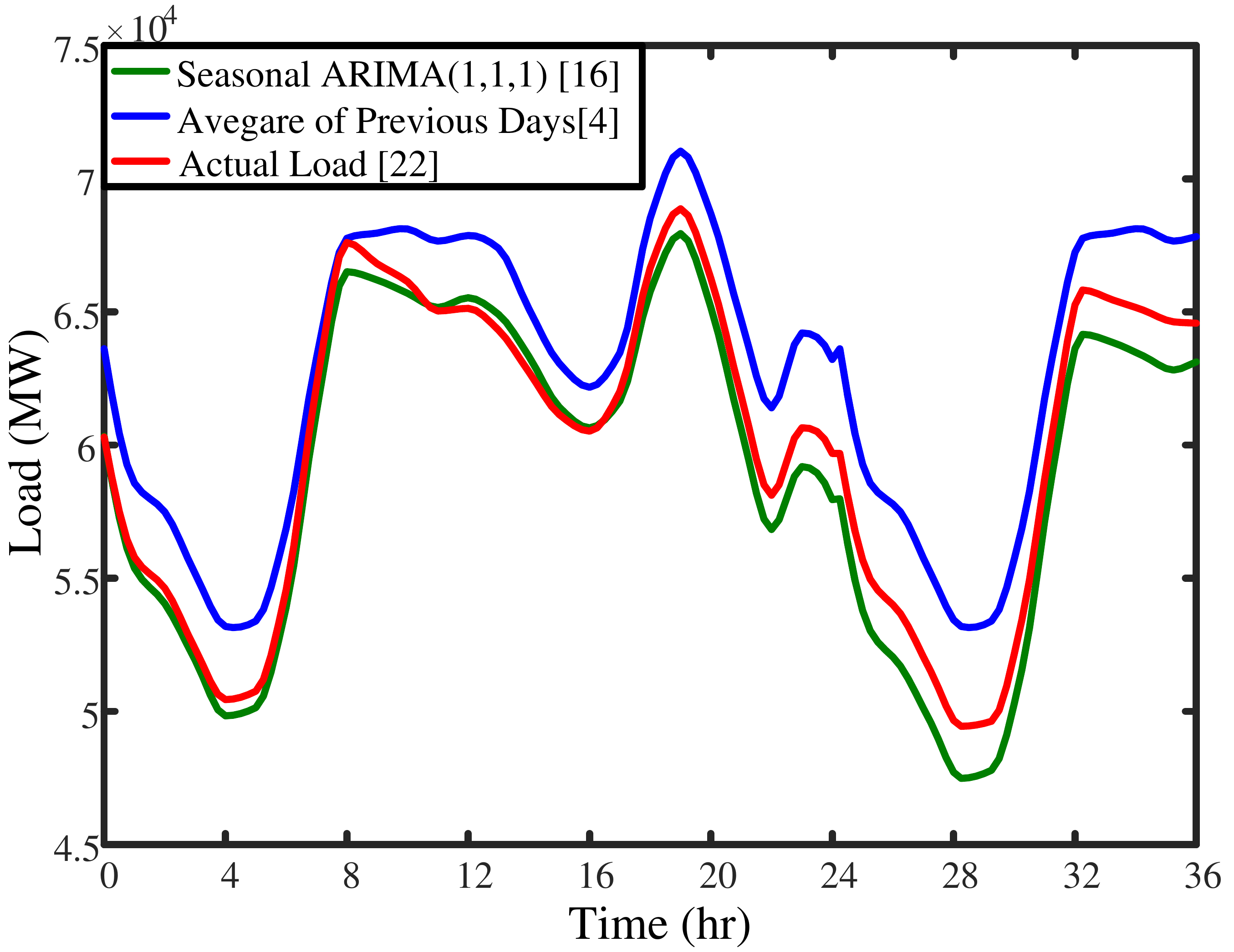} }%
\caption{Load forecasting result for ARIMA \cite{2007-taylor-loadpred-EUdata} and the average from 11/30 to 12/16 \cite{2012-He-EV-selection}.}\label{fig:load_1217_1218}
\end{center}
\end{figure}

\subsection{Evaluation in Terms of Charging Cost}

We evaluate the charging costs under different numbers of EVs in the charging stations.
The corresponding results for different algorithms are shown in Fig.~\ref{fig:charging_cost_compare}, where the charging costs are normalized. Among the algorithms, the method in \cite{2012-wen-EV-selection} has a much higher charging cost because this algorithm only considers user convenience and thus uses as much power as possible to reach the maximum user convenience.
The charging cost is considered in the three other algorithms, and so the charging cost is much lower than the method in \cite{2012-wen-EV-selection}.
The cost difference between the method in \cite{2012-He-EV-selection} and CSA is due to forecasting accuracy.
Meanwhile, a $2\%$ difference in cost is observed between CSA and DCSA.
This is due to the fact that the algorithm LCCM used in DCSA provides a sub-optimal solution under the limitations of constraints.
However, if there are more than $2,000$ EVs are in the charging stations, the charging cost of CSA would be unavailable due to the high computational complexity.

\begin{figure}
\begin{center}
\resizebox{2.7in}{!}{%
\includegraphics*{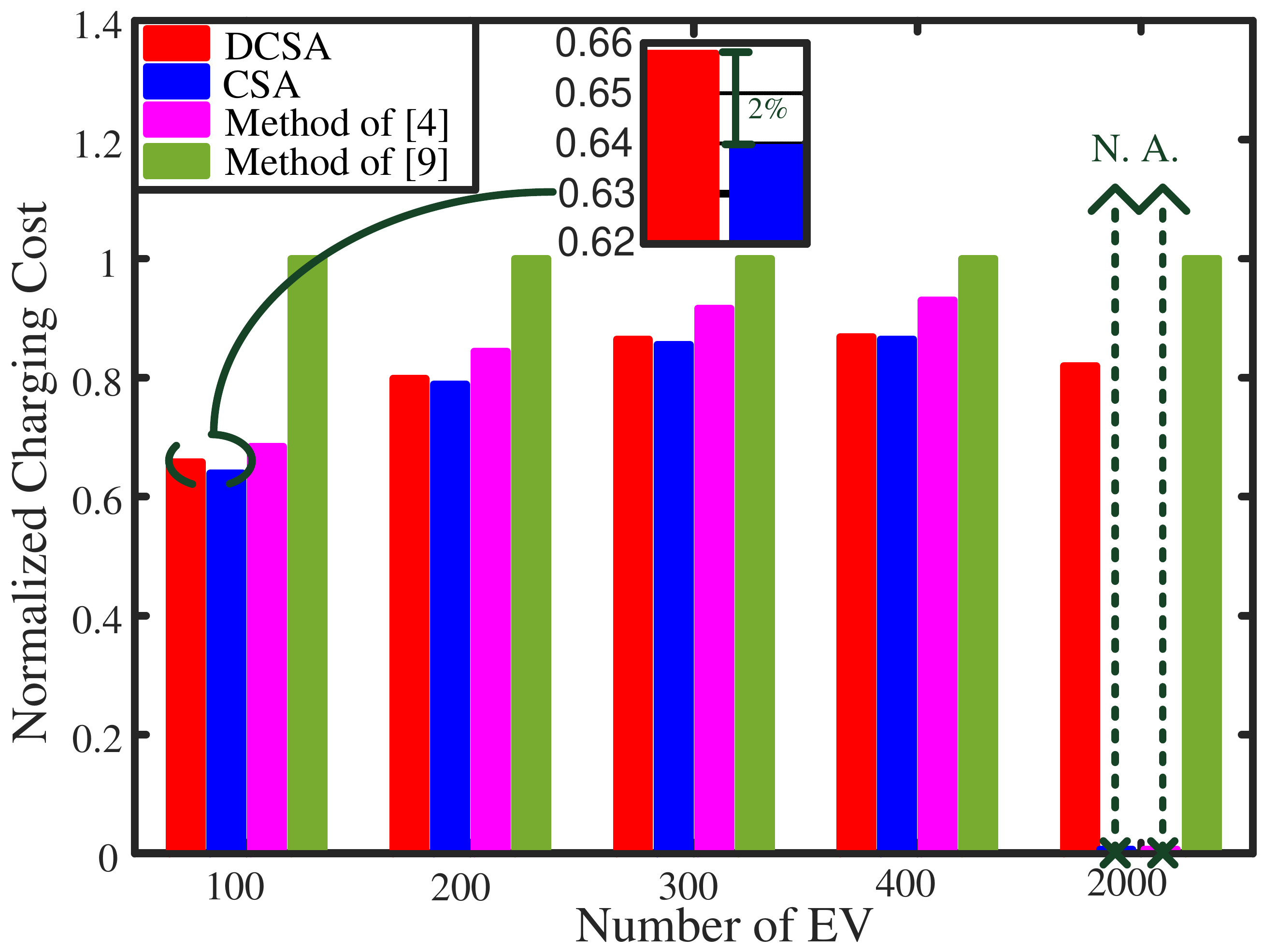} }%
\caption{Charging cost comparison under different methods.}\label{fig:charging_cost_compare}
\end{center}
\end{figure}

\subsection{Evaluation in Terms of Charging Time}

As mentioned in Section \ref{subsec:weight}, the proposed user convenience metric can effectively reduce charging time.
Therefore, the four described algorithms are compared in terms of the average charging time when the total number of EVs is $200$ EVs.
Charging time is defined as the difference between the arrival time and the finish time for each EV.
A fixed peak constraint has to be satisfied in all of the algorithms.
The comparison of the charging time is provided in Fig.~\ref{fig:EV200_waiting_time_compare}.

According to Fig.~\ref{fig:EV200_waiting_time_compare}, the average charging time for CSA is $9.08$ hours and it is $8.88$ hours for DCSA.
However, the average charging times of cost minimization algorithm and user-convenience maximization algorithm are $12.74$ and $12.88$ respectively.
Hence, the two proposed methods show a significant improvement of $30\%$ compared to the other two methods.
This result indicates the effectiveness of the proposed user convenience metric in (\ref{eq:weight}).
A gap exists between DCSA and CSA because finding the charging cost with DCSA is higher than with CSA; hence, the charging time is shorter.
Indeed, as explained in the previous subsection, the charging cost of using DCSA is slightly higher than those incurred with CSA.
The average charging time of the methods in \cite{2012-He-EV-selection} and \cite{2012-wen-EV-selection} are very close to each other; neither of the two methods has a specific strategy to reduce the charging time.

\begin{figure}
\begin{center}
\resizebox{3in}{!}{%
\includegraphics*{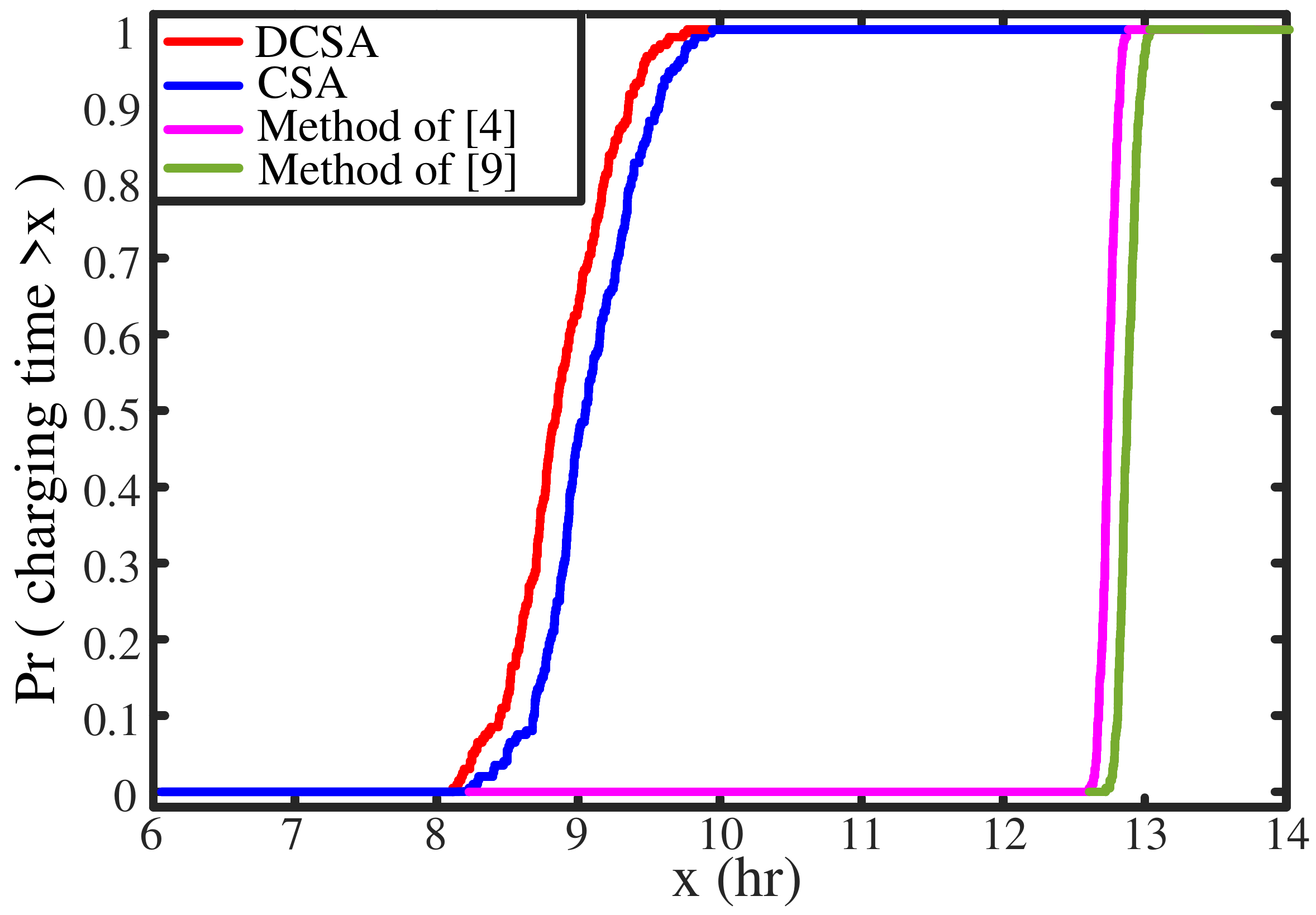} }%
\caption{Cumulative distribution function of the average charging time of different algorithms.}\label{fig:EV200_waiting_time_compare}
\end{center}
\end{figure}

The charging profile of the $10$-th EV in the charging station is provided in Fig.~\ref{fig:charge_profile_compare}.
In line with the proposed user convenience metric, this EV is selected for charging with the maximum charging rate and charging takes only a few time slots for both CSA and DCSA; hence, the charging time is reduced.
The method in \cite{2012-He-EV-selection} distributes the demand over the charging period, and so the charging task is finished when the time reaches the deadline.
The charging profile of the method in \cite{2012-wen-EV-selection} is periodical because its designed incorporation of user convenience.
Consequently, the proposed user convenience metric can consider the charging time in the formulation.

\begin{figure}
\begin{center}
\resizebox{3in}{!}{%
\includegraphics*{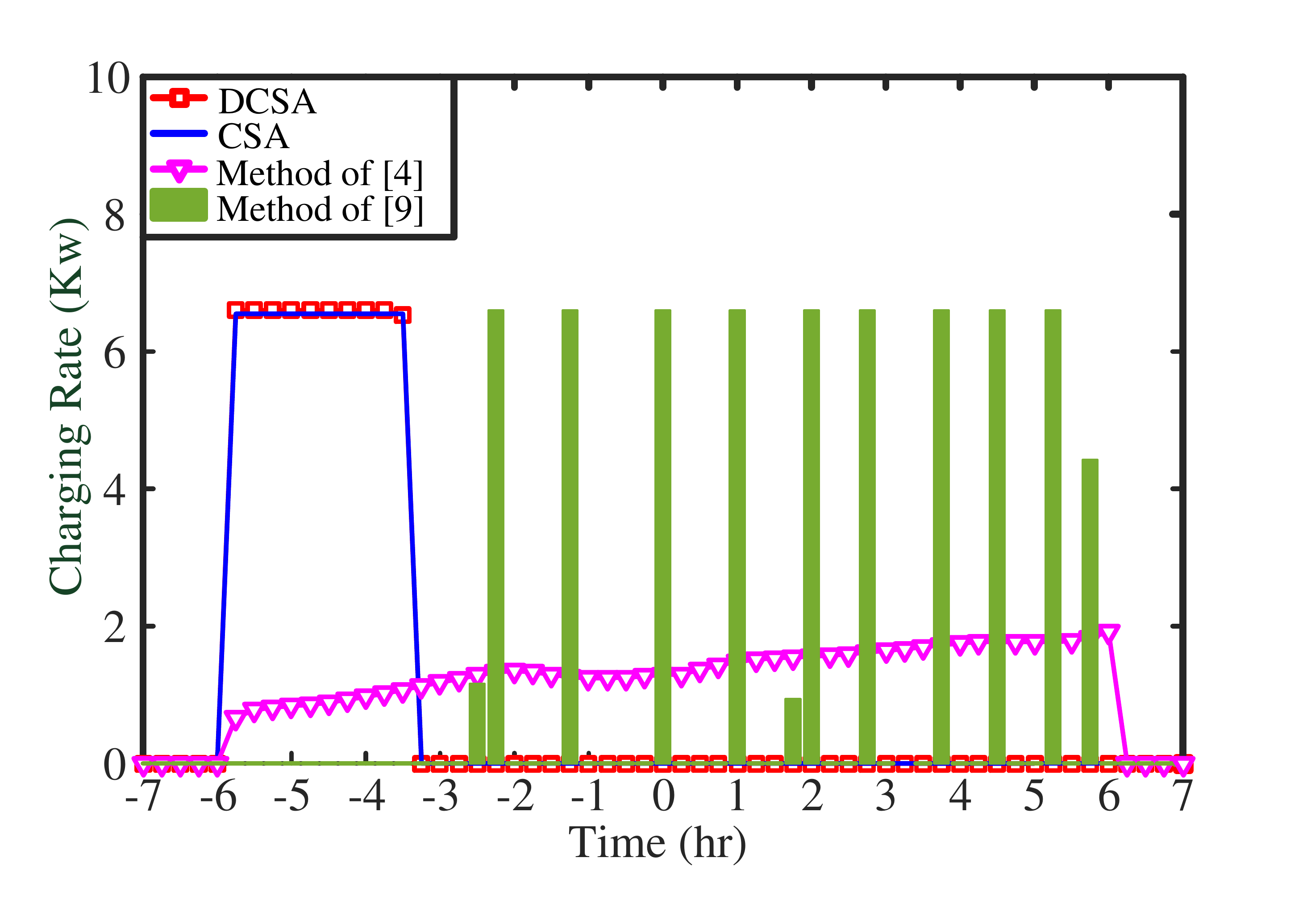} }%
\caption{Charging profile of the 10th EV under different methods.}\label{fig:charge_profile_compare}
\end{center}
\end{figure}

\subsection{Evaluation in Terms of Computation Time}
In Section \ref{sec:Low_com_dis_alg}, we present a low-complexity distributed algorithm and claim that the computation complexity is very low.
To prove this claim, we simulate $200$ and $2,000$ EVs in terms of computation time and charging costs using CSA and DCSA then the results are listed in Table \ref{tb:cost_bisec}.

\begin{table} \small
\renewcommand{\arraystretch}{1.3}
\begin{center}
\caption{Computation time comparison for finding the minimum charging cost}\label{tb:cost_bisec}
\begin{tabular}{|C{2.5cm}|c|c|c|c|c|c|c|c|}
\hline
Scenario    							&     				  &  CSA  		& DCSA \\
\cline{1-4}
\multirow{1}{*}{200 EVs }               & Time (s)             & $49.45$     & $ 8.57 \times 10^{-4} $  \\
\cline{2-4}
\multirow{1}{*}{9,246 Constraints }      & Cost (C$\$$)       & $27.65$     & $\phantom{0}27.85$        \\
\hline \hline
\multirow{1}{*}{2,000 EVs }              & Time (s)             & N.A.        & $ 8.23 \times 10^{-4} $   \\
\cline{2-4}
\multirow{1}{*}{96,048 Constraints }     & Cost (C$\$$)       & N.A.        & $2239.90$        \\
\hline

\end{tabular}
\end{center}
\end{table}

For the case with $200$ EVs, the computation time for DCSA is $8.57 \times 10^{-4} $ second which indicates about $57,700$ times faster than that using the interior-point method (the CSA).
The proposed algorithm has very similar running times for $200$ and $2,000$ EVs.
However, the interior-point method is unavailable for ${2,000}$ EVs due to the $10$ times greater number of the constraints.

\begin{table} \small
\renewcommand{\arraystretch}{1.3}
\begin{center}
\caption{Computation time comparison for finding the minimum user convenience level}\label{tb:uc_bisec}
\begin{tabular}{|C{2cm}|c|c|c|c|c|c|c|c|}
\hline
Scenario    							&      				  &  UCM  		& DUCM \\
\cline{1-4}
\multirow{2}{*}{200 EVs }               & Time (s)             & $ 7.83 \times 10^{-3} $       & $ 5.31 \times 10^{-2} $  \\
\cline{2-4}
                                        & $u_{\min}$           & $ 2.20 \times 10^{-3} $       & $ 2.20 \times 10^{-3} $   \\
\hline \hline
\multirow{2}{*}{{2,000} EVs }           & Time (s)             & $ 3.08 \times 10^{-2} $       & $ 5.46 \times 10^{-1} $    \\
\cline{2-4}
                                        & $u_{\min}$           & $ 2.02 \times 10^{-3} $       & $ 2.02 \times 10^{-3} $    \\
\hline

\end{tabular}
\end{center}
\end{table}

Table \ref{tb:cost_bisec} shows a comparison of the first phase of the two algorithms.
The second phase is compared next.
The DCSA results in Table \ref{tb:cost_bisec}, utilized to compare UCM and DUCM, as listed in Table \ref{tb:uc_bisec} and illustrated in Fig.~\ref{fig:iter_compare_UC}.
If the available power is the same in the simulation, the distributed algorithm can obtain the same performance as the centralized one.
Moreover, DUCM can convergence after $10$ iterations between all SAs.
According to the results, DUCM needs more time for ${2,000}$ EVs due to the iteration between all SAs; however, the distributed algorithm can still make a decision under one second with ${2,000}$ EVs in the charging station.
The proposed distributed algorithm thus significantly reduces the computation time.

\begin{figure}
\begin{center}
\resizebox{3in}{!}{%
\includegraphics*{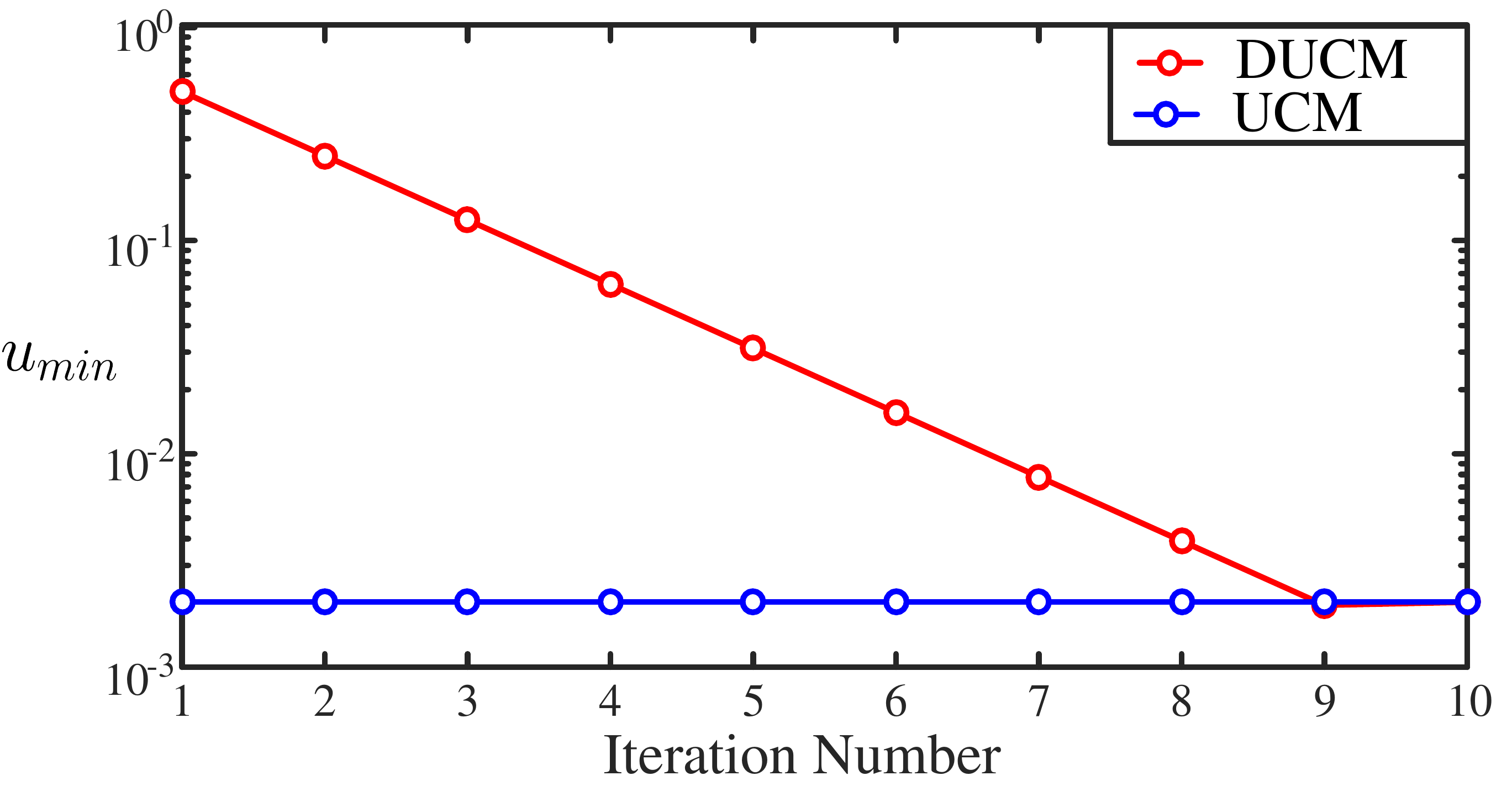}}%
\caption{Iteration of DUCM and its convergence with UCM.}\label{fig:iter_compare_UC}
\end{center}
\end{figure}

\subsection{Comparisons with different methods}
With the exception of centralized and decentralized algorithms, some interesting algorithms still remain.
In the simulations, we compare four algorithms, namely DCSA (the proposed low-complexity distributed method), the method of \cite{2015-tan-queue} (queue method), the method of \cite{2014-prob-EVselect} (stochastic optimization), and the method of \cite{2015-mpc-EVselect} (model predictive control approach).
To compare the charging cost equally, the cost function is replaced with (\ref{eq:obj1}).
The number of EVs is set to $100$.
We provide the simulations by comparing the total charging cost, total user convenience, average charging time, and peak load after scheduling. 
Also, user convenience is a parameter scaled from $0$ to $1$.
Therefore, we give the EVs, which finish the charging tasks and stay in the charging station before the deadline, with maximum user convenience, which is $1$.
By doing so, if the EVs finish the charging tasks as soon as possible, then EV owners can obtain the maximum user convenience.
This idea helps us show the advantage of reducing the average charging time.
The comparisons are in different scales.
Thus, we normalize them with the largest value.
The simulation results are provided in Fig.~\ref{fig:compare_with_other}.

The simulation result shows that the total charging cost of DCSA is almost the same as that of the method of \cite{2015-mpc-EVselect} because both works focus on charging-cost minimization with the objective function.
With this approach, the peak load after charging obtains the same result. 
User convenience is designed with the remaining time slots and the minimum number of time slots required to finish the charging. 
By contrast, in the method of \cite{2015-mpc-EVselect}, EVs are selected only based on the user-defined deadline, which can be regarded as the same as the remaining time slots.
As our work considers more parameters than the method of \cite{2015-mpc-EVselect}, the total user convenience of the proposed method is slightly higher.
The objective function listed in \cite{2015-tan-queue} has a trade-off term, and thus the method of \cite{2015-tan-queue} cannot obtain the minimum charging cost.
The method of \cite{2014-prob-EVselect} does not consider the cost.
Thus, the charging cost has the maximum value. 
Then, both \cite{2015-tan-queue} and \cite{2014-prob-EVselect} do not consider charging time and user convenience, and their performance cannot approach the results of the decentralized algorithm and the method of \cite{2015-mpc-EVselect}. 
Given that the current work also considers charging time with user convenience, maximum user convenience and minimum average charging time are utilized.

\begin{figure}
\begin{center}
\resizebox{3.2in}{!}{%
\includegraphics*{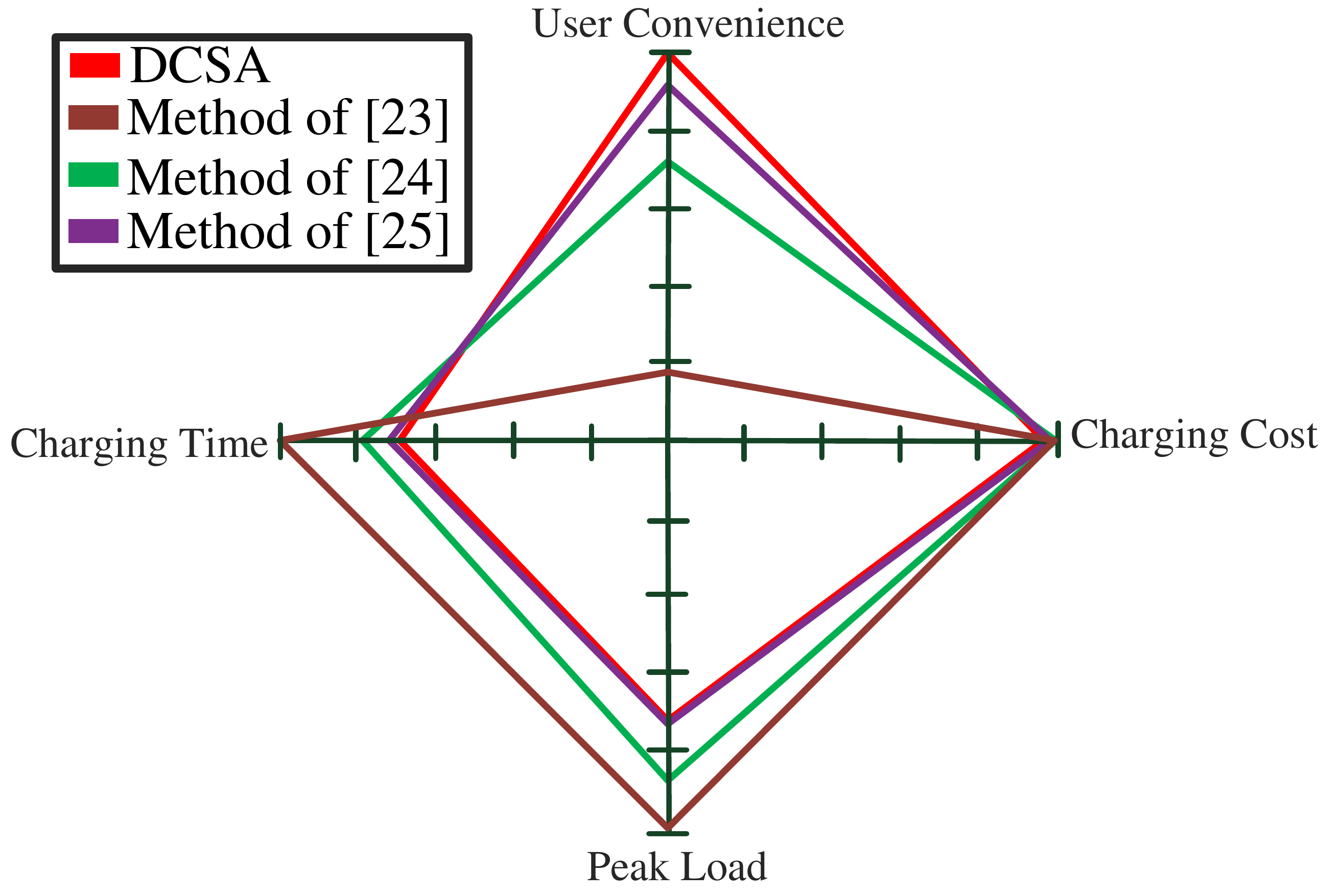} }%
\caption{Comparison with other works.}\label{fig:compare_with_other}
\end{center}
\end{figure}

\subsection{Evaluation in Terms of Data Rate}

\begin{table*}
\begin{center}
\caption{Transmission and reception comparison of centralized and distributed method}\label{tb:data_rate_compare}
\begin{tabular}{|c|c|c|c|c|c|c|}
  \hline
    \multirow{2}{*}{Communication} &\multicolumn{2}{c|}{CSA}     &\multicolumn{2}{c|}{DCSA}  \\
    \cline{2-5}
						    & Transmission    & Reception          & Transmission   & Reception\\
    \hline
    At EV  					& $2$ to SA       & $1$ from SA        & $2$ to SA      & $1$ from SA \\
    \hline
    \multirow{3}{*}{At SA}  & $1$ to each EV  &  $2$ from each EV  & $1$ to each EV &  $2$ from each EV\\

     &  $7 |{\cal H}_{m, t}|$ to CA  &   $ |{\cal H}_{m, t}|$ from CA  & $|{\cal W}_{m, t}| +1$ to CA  & $1$ from CA\\

     &                          &                            & $4$ to next SA  & $4$ from previous SA \\
    \hline
    At CA & $ |{\cal H}_{m, t}|$ to  SA  &  $7|{\cal H}_{m, t}|$ form SA  & 1 to each  SA  &  $|{\cal W}_{m, t}|+1$ from each SA\\

    \hline
    \multirow{3}{*}{Total}  &\multicolumn{2}{c|}{ $ \sum_{m=1}^{M} |{\cal H}_{m, t}| $ ~(between SA and EV)}          &\multicolumn{2}{c|}{ $ \sum_{m=1}^{M} |{\cal H}_{m, t}| $ (between SA and EV)}    \\
						    &\multicolumn{2}{c|}{ $ \sum_{m=1}^{M} 8|{\cal H}_{m, t}| $ (between CA and SA)}          &\multicolumn{2}{c|}{  $\sum_{m=1}^{M} (|{\cal W}_{m, t}| +2 )$ (between CA and SA)}     \\
						    &\multicolumn{2}{c|}{ }                                &\multicolumn{2}{c|}{ $8Ma$ (between SA and SA)}    \\
    \hline
\end{tabular}
\end{center}
\end{table*}

The numbers of transmitted messages between the centralized and distributed algorithms are compared in Table \ref{tb:data_rate_compare}, in which $|{\cal A}|$ denotes the cardinality of set ${\cal A}$.
For the two algorithms, the variables transmitted from EVs to SAs are identical; $r_{i}$ and ${\tt SOC}_{i}^{\rm fin}$.
Then, EV $i$ receives the decision of the charging rate at time slot $t$, $P_{i, t}$, from the SA.
Therefore, the difference of the two algorithms lies in the interaction between the CA and the SAs.

For the CSA, each SA sends $u_{i, t}$, $P_{i_{\max}}$, $P_{i_{\min}}$, $r_{i}$, $E_{i}^{\rm ini}$, $E_{i}^{\rm fin}$, and $E_{i}^{\rm cap}$ to the CA; thus, the total amount is $7 \times|{\cal H}_{m, t}|$.
After obtaining the scheduling result, the CA broadcasts $P_{i, t}$ to SAs, and then each SA receives the amount of $|{\cal H}_{m, t}|$.
Therefore, the total number of messages exchanged between the CA and the SAs is $ \sum_{m=1}^{M} 8|{\cal H}_{m, t}| = 8N$.
The number of transmitted units can then be denoted as $O(N)$, which means the number of transmitted units mostly depends on the total number of EVs.

In the DCSA, the $m$th SA sends the demand $d_{m}$ and its time window information to the CA.
The CA then broadcasts the available power to all the SAs.
Therefore, the total amount of data exchanged between the CA and the SAs is $\sum_{m=1}^{M} (|{\cal W}_{m, t}| +2 )$.
After obtaining the available power, the SAs use DUCM to determine the charging decision.
Each SA receives the $u_{\min}$, $P^{\rm char}$, $b^{\rm char}$, and $i^{\rm char}$ values from its previous SA.
Next, the new variables are recalculated and passed to the next SA.
Therefore, each SA needs to transmit and receive four data; the total amount of data in one iteration then becomes $8M$.
Assuming that the iteration converges in $a$ time, the total number of transmitted messages during the iteration is approximately $8Ma$.
Hence, the total amount of data transmitted for the DCSA is $\sum_{m=1}^{M} (|{\cal W}_{m, t}| +2 )+8Ma$.
The number of transmitted units can then be denoted as $O\left( \max  (|{\cal W}_{m, t}|, M) \right)$, which means the number of transmitted units mostly depends on the maximum value of the number of charging stations and the length of the sliding time window.

We studied a case with $200$ EVs randomly distributed among $6$ SAs.
We assumed that the time window length for each SA is $96$ for the worst case, that the EVs stay for a whole day in a charging station, and the iteration number of DUCM is $10$, as shown in Fig.~\ref{fig:iter_compare_UC}.
Table \ref{tb:data_rate_compare} shows that the CSA needs to transmit $1,600$ units of data for each time slot; however, this amount is only $1,068$ for the DCSA.
Thus, the total amount of data required for both algorithms is approximately $33.25\%$ less with the DCSA compared with the CSA.
The reduction can vary across different settings and cases.
However, in most of the scenarios, the number of EVs, $N$, is greater than the number of charging stations, $M$, and the length of the sliding time window, $|{\cal W}_{m, t}|$.
Therefore, the proposed distributed algorithm, DCSA, can still achieve a considerable reduction in the number of transmitted units.

\section{Conclusion}\label{sec:conclusion}

This work studied the EV charge scheduling problem.
Unlike previous studies that focused only on charging-cost minimization or user-convenience maximization, we considered both factors and proposed an efficient centralized scheduling mechanism to solve the formulated bi-objective optimization problem.
The proposed method can simultaneously obtain the minimum charging cost and reduce the charging time.
Although the centralized scheduling algorithm can guarantee good performance, this algorithm may have a high computational complexity and data transmission rate.
Therefore, a low-complexity distributed algorithm was proposed to obtain a performance level comparable to that of the centralized algorithm, which also significantly reduces the number of transmitted messages.

{\renewcommand{\baselinestretch}{1}
\begin{footnotesize}
\bibliographystyle{IEEEtran}
\bibliography{References_SmartGrid.bib}
\end{footnotesize}}

\appendix

\subsection{Proof of Pareto Optimality}\label{subsec:proof_optimal}
In this Appendix, we prove that the solution of CSA is Pareto-optimal.
The solution solved by the proposed method is $( z^{*}_{t}, \qp^{*})$ and suppose there exists another solution, $(\bar{z_{t}}, \bar{\qp})$, can reach lower charging cost and higher user convenience.

Because $\mathcal{P}_1$ is proved to be convex, the charging cost reaches optimal then its corresponding load, $z^{*}_{t}$, is also optimal value.
Therefore, $z^{*}_{t}$ is equal to $\bar{z_{t}}$.
Furthermore, we only have to show why UCM can obtain higher user convenience.

Putting (\ref{eq:weight}) into $\mathcal{P}_2$, the objective function becomes
\begin{equation}
\max_{\qP} \sum_{t \in \qw_{t}} \sum_{i \in {\cal H}_{t}} \frac{ P_{i_{\max}}}{ \left( {\tt SOC}_{i}^{\rm fin} - \left( {\tt SOC}_{i, 1} + \sum_{t=1}^{t-1} \frac{P_{i, t}}{E_{i}^{\rm cap}} \right) \right) E_{i}^{\rm cap} w_{i, t}  }.
\end{equation}
Assuming $E_{i}^{\rm cap}$  and $ P_{i_{\max}}$ are constant for every EV.
In order to maintain the maximum value, the EV with higher user convenience will receive the charging first, since $w_{i, t}$ is a decreasing value, making  user convenience higher after charging.
The value of ${\tt SOC}_{i, 1}  + \sum_{t=1}^{t-1} \frac{P_{i, t}}{E_{i}^{\rm cap}}$ should therefore approach ${\tt SOC}_{i}^{\rm fin}$ as soon as possible.
Hence, the charging rate,  $P_{i, t}$, must be the maximum charging rate.
From the proof, $\bar{\qs}$ and $\bar{\qp}$ are equal to $\qs^{*}$ and $\qp^{*}$, respectively.
Therefore, the solution of our proposed algorithm is the Pareto-optimal solution.

On the basis of the previous proof, we can explain the Pareto optimality with a graph that is used to find the Pareto front.
Moreover, we can show the proposed algorithm can find the desired Pareto-optimal solution.
Assuming there are total $100$ EVs, we attempt to find all possible solutions according to the two objective functions.
The user convenience of the EVs, which complete the charging tasks and remain in the charging station before the deadline, is determined to be $1$.
In Fig. \ref{fig:pareto_ana}, blue and gray areas represent all the possible solutions for the EV charging.
The blue line can then be regarded as the Pareto front.
In our formulation, the constraint in (\ref{eq:total_for}) shows that we have to meet the requirement of EVs when they leave, which makes it a bound constraint.
Therefore, only the blue area are the feasible solutions for the formulation.

\begin{figure}
\begin{center}
\resizebox{3.5in}{!}{%
\includegraphics*{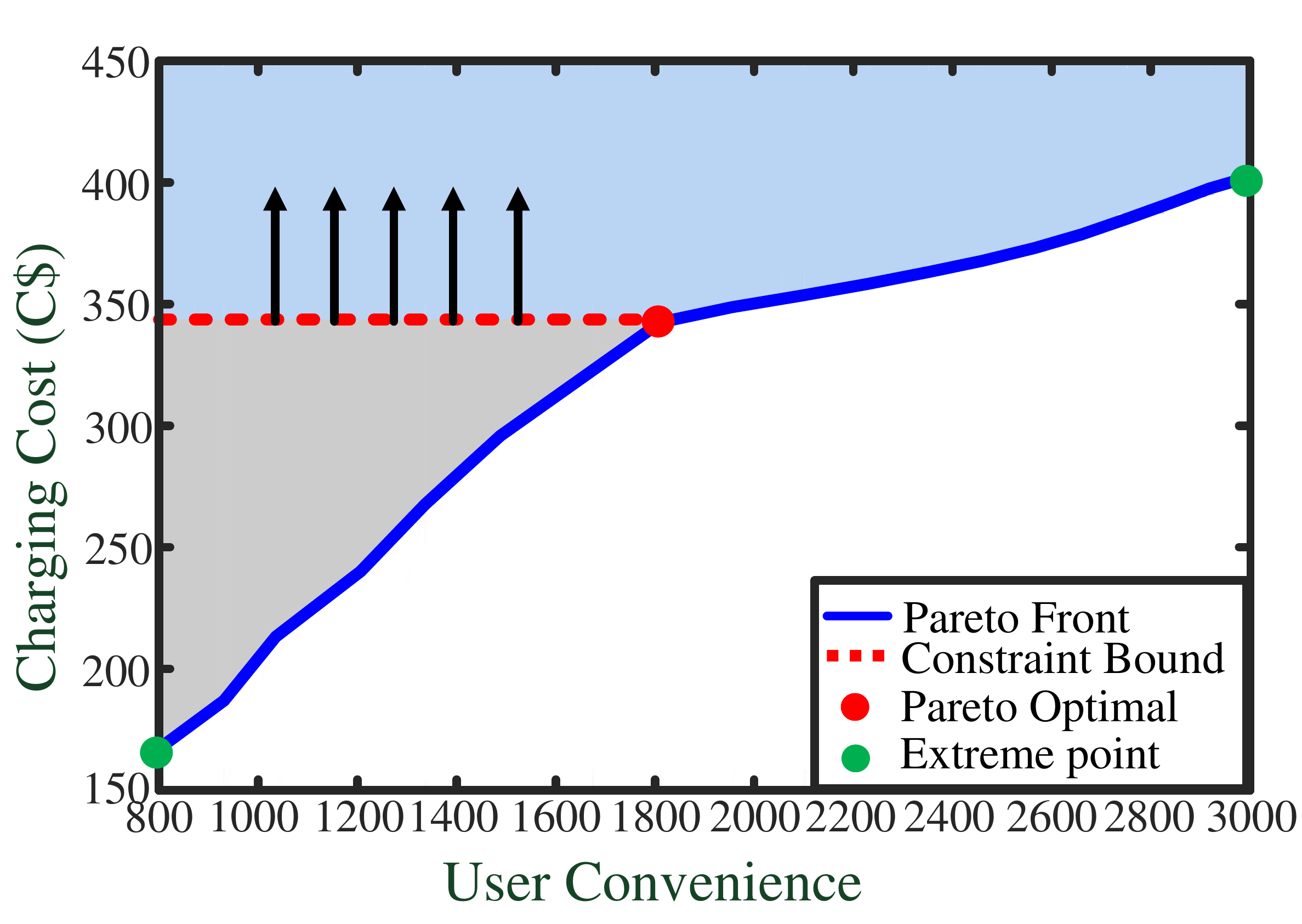} }%
\caption{Pareto front and Pareto-optimal solution. }\label{fig:pareto_ana}
\end{center}
\end{figure}

In the blue area, we want to find the solution with minimum charging cost and the maximum user convenience at the same time.
However, it is difficult to find the optimal solution for both objective functions.
Therefore, we consider charging-cost minimization initially because we stand at the side of CA to make the decision.
A dominated solution such exists for the charging cost.
Next, among all the possible charging patterns under the given charging cost, we select the one that can achieve maximum user convenience.
By doing so, we choose the red point as the desired Pareto-optimal solution and it is also the solution solved by the proposed algorithm.
However, with the exception of the red point, some interesting solutions in the Pareto front still remain.
Other preference of charging stations or EV owners can be considered, so that other solutions can be explored.

Other than the desired Pareto-optimal solution, there are two extreme solutions denoted with green points in Fig. \ref{fig:pareto_ana}.
The values for the extreme solutions are $(800,170)$ and $(3000,400)$, respectively.
The $(800,170)$ cannot meet user requirement and thus cannot be the solution candidate.
The $(3000,400)$ can be regarded as the charge scheduling without control.
That is, when the EV owners plug in the charging socket, the EVs can receive the power immediately.
This approach allows the EV owners to obtain the highest user convenience, but can also increase the charging cost to CA.
Therefore, we do not consider two extreme solutions as the solution candidates.

\subsection{Proof of Lemma 1}\label{subsec:proof_optimal_cost}
The objective function with the assumption is shown in (\ref{eq:int_cost}), in which $\left( k_{0} L_t ^{{\rm base}}+\frac{k_{1}}{2} {L_t^{\rm base}}^{2} \right)$ is a constant when the base load is determined.
Then, to obtain the maximum user convenience, all the available power will be used to charge EVs.
Therefore, we can just consider the first two terms and replace $z_{i, t}$ with $L_t^{\rm base} + \sum_{i \in {\cal H}_{t}} P_{i, t}$ as
\begin{equation} 
\sum_{t=1}^{T} \left( k_{0} \left( L_t^{\rm base}+\sum_{i \in {\cal H}_{t}} P_{i, t} \right)+\frac{k_{1}}{2} \left(L_t^{\rm base} + \sum_{i \in {\cal H}_{t}} P_{i, t}\right)^{2}\right) .
\end{equation}
We use $ \sum_{t=1}^{T} f( L_t^{\rm base} + \sum_{i \in {\cal H}_{t}} P_{i, t})$ to denote the above function.
Applying Jensen's inequality, $f \left(E[ L_t^{\rm base} + \sum_{i \in {\cal H}_{t}} P_{i, t}] \right) \leq E[f( L_t^{\rm base} + \sum_{i \in {\cal H}_{t}} P_{i, t})]$ , to the function, we obtain
\begin{align}
 & f \left( \frac{( L_1^{\rm base} + \sum_{i \in {\cal H}_{t}} P_{i, 1}) + \cdots + ( L_T^{\rm base} + \sum_{i \in {\cal H}_{t}} P_{i, T})}{T}\right)  \notag\\
 & \leq \frac{f( L_1^{\rm base} + \sum_{i \in {\cal H}_{t}} P_{i, 1}) + \cdots + f( L_T^{\rm base} + \sum_{i \in {\cal H}_{t}} P_{i, T})}{T}.
\end{align}
The equality holds when
\begin{align}
& L_1^{\rm base} + \sum_{ i \in {\cal H}_{t}} P_{i, 1} = \cdots = L_T^{\rm base} + \sum_{i \in {\cal H}_{t}} P_{i, T} \notag \\
& = \frac{( L_1^{\rm base} + \sum_{i \in {\cal H}_{t}} P_{i, 1}) + \cdots + ( L_T^{\rm base} + \sum_{i \in {\cal H}_{t}} P_{i, T})}{T}.
\end{align}
Therefore, the optimal solution is represented as
\begin{equation} \label{eq:jenson_solution}
\sum_{i \in {\cal H}_{t} } P_{i, t} = \frac{\sum_{t=1}^{T} \left(  L_t^{\rm base} + \sum_{i \in {\cal H}_{t}} P_{i, t}  \right)    }{T} - L_t^{\rm base}, \, \forall t.
\end{equation}

\end{document}